\documentclass[review]{elsarticle}
\usepackage{amsmath,amsfonts,amssymb,amsthm}
\usepackage{url}

\newtheorem{theorem}{Theorem}

\newtheorem{lemma}[theorem]{Lemma}
\newtheorem{proposition}[theorem]{Proposition}

\theoremstyle{remark}

\theoremstyle{definition}
\newtheorem{question}{Question}

\DeclareMathOperator{\nsc}{inrc}

\begin{document}

\begin{frontmatter}
\title{Initial non-repetitive complexity of infinite words}

\author{Jeremy Nicholson\fnref{fn1}}
\ead{jnich998@hotmail.com}

\author{Narad Rampersad\corref{cor1}\fnref{fn2}}
\ead{n.rampersad@uwinnipeg.ca}

\address{Department of Mathematics and Statistics,
University of Winnipeg,
515 Portage Avenue,
Winnipeg, Manitoba R3B 2E9 (Canada)}

\cortext[cor1]{Corresponding author}
\fntext[fn1]{The author was supported by an NSERC USRA.}
\fntext[fn2]{The author is supported by an NSERC Discovery Grant.}

\begin{abstract}
The \emph{initial non-repetitive complexity function} of an infinite word
${\bf x}$ (first introduced by Moothathu) is the function of $n$ that
counts the number of distinct factors of length $n$ that appear at the beginning
of ${\bf x}$ prior to the first repetition of a length-$n$ factor.  We
examine general properties of the initial non-repetitive complexity function,
as well as obtain formulas for the initial non-repetitive complexity of the
Thue--Morse word, the Fibonacci word and the Tribonacci word.
\end{abstract}


\end{frontmatter}

\section{Introduction}

For any infinite word ${\bf x}$, there is an associated
\emph{complexity function} $c_{\bf x}$ defined as follows:  the
quantity $c_{\bf x}(n)$ is the number of distinct factors of length
$n$ that appear in the word ${\bf x}$.  Properties of the complexity
function for various classes of infinite words have been extensively
studied \cite[Chapter~4]{BR10}.  Several variants of the complexity function have
been introduced and studied, such as \emph{palindrome complexity}
\cite{ABCD03} or \emph{abelian complexity} \cite{RSZ10}.  In this
paper we study the \emph{initial non-repetitive complexity function}, which
was first introduced by Moothathu \cite{Moo12}.

We define the \emph{initial non-repetitive complexity function} $\nsc_{\bf x}(n)$ for
an infinite word ${\bf x} $ by
$$\nsc_{\bf x}(n)=\max\{m\in \mathbb{N}:x_i\cdots x_{i+n-1}\neq
x_j\cdots x_{j+n-1} \text{ for every } i,j \text{ with } 0\leq i<j\leq m-1\}.$$
In other words $\nsc_{\bf x}(n)$ is the maximum number of length-$n$
factors that we see when reading ${\bf x}$ from left to right prior to the
first repeated occurrence of a length-$n$ factor.

Moothathu posed the following question in his paper:  Is it possible
to get some idea about the topological entropy of a dynamical system by
looking only at initial segments of the orbit of some point?  As an
attempt to answer this question, he introduced the quantity
\[
\limsup_{n \to \infty} \frac{\log \nsc_{\bf x}(n)}{n},
\]
which he called ``non-repetitive complexity''.  In this paper, we will
use the term \emph{non-repetitive complexity} for the following function:
$$\mbox{nrc}_{\bf x}(n)=\max\{m\in \mathbb{N}: \exists k,\; x_i\cdots x_{i+n-1}\neq
x_j\cdots x_{j+n-1} \text{ for every } i,j \text{ with } k\leq i<j\leq k+m-1\}.$$

This paper is primarily about the function $\nsc_{\bf x}(n)$.
Although Moothathu introduced the concept, he did not explicitly
compute this function for any particular infinite words.  In a future
work, it would be of interest to study the function $\mbox{nrc}_{\bf
  x}(n)$, which likely has many similar properties.

The initial non-repetitive complexity also bears some resemblance to the
quantity $R'_{\bf x}(n)$, which is the length of the shortest prefix
of ${\bf x}$ that contains at least one occurrence of \emph{every}
length-$n$ factor of ${\bf x}$ \cite{AB95}.  There is also a
connection (which we shall make use of later) to the concept of a
word with \emph{grouped factors}, which was studied by Cassaigne
\cite{Cas97}.

In the remainder of this paper, we will give some general properties
of the initial non-repetitive complexity function in comparison to the usual
complexity function.  We will also give explicit formulas for the
initial non-repetitive complexity of some of the classical infinite words, namely,
the Thue-Morse word (${\bf m}$), the Fibonacci word (${\bf f}$) and 
the Tribonacci word (${\bf t}$).  Finally we examine the possible range of
values that the initial non-repetitive complexity function can take for
squarefree words.  We attempt to construct squarefree
words with slowly growing initial non-repetitive complexity
functions.  This is somewhat similar to the notion of a \emph{highly
  repetitive word} \cite{RV13}.

\section{Preliminaries}

Let $\Sigma$ denote a finite \emph{alphabet} and let $\Sigma^*$ denote
the set of finite words over $\Sigma$.  
Let $\{0, 1\}$ be the alphabet in the case of the Thue--Morse and
Fibonacci words and let the alphabet be $\{0, 1, 2\}$ for the
Tribonacci word. If $\theta : \Sigma^* \to \Sigma^*$ is a morphism, then
$\theta^r(u)$ for a non-negative integer $r$ and a word $u$ is
obtained by applying the morphism $\theta$ to $u$ $r$ times (we define
$\theta^0(u)=u$). By convention, we denote the string of length 0 by
$\epsilon$.  A word $y$ is a \emph{factor} of a word $x$ if $x$ can be
written as $x = uyv$ for some words $u$ and $v$.
If $x$ is a word (finite or infinite) we let $x[i\ldots
j]$ denote the factor of $x$ of length $j-i+1$ that starts at position $i$
in $x$.  We denote the length of any finite word
$u$ by $|u|$.  For any letter $a$, we denote the number of occurrences
of $a$ in $u$ by $|u|_a$.

A word $x = x_1\cdots x_n$ has \emph{period} $p$ if $x_i=x_{i+p}$ for
$i = 1,\ldots,n-p$.  An infinite word ${\bf w}$ is \emph{ultimately
  periodic} if ${\bf w} = uvvvv\cdots$ for some words $u$ and $v$.  If
$u = \epsilon$ then ${\bf w}$ is \emph{periodic}.  If ${\bf w}$ is
not ultimately periodic then it is \emph{aperiodic}.  If every factor
of ${\bf w}$ occurs infinitely often in ${\bf w}$ then ${\bf w}$ is
\emph{recurrent}.

If ${\bf w} = a\theta(s)\theta^2(s)\theta^3(s)$, where $\theta :
\Sigma^* \to \Sigma^*$ is a morphism, $a \in \Sigma$, $s \in
\Sigma^*$, and $\theta(a) = as$, then ${\bf w}$ is \emph{pure
  morphic}.  The \emph{adjacency matrix} associated with a morphism
$\theta$ is the matrix $M$ with rows and columns indexed by elements
of $\Sigma$ such that the $ij$ entry of $M$ equals $|\theta(j)|_i$.

A \emph{square} is a non-empty word of the form $xx$, and a
\emph{cube} is a non-empty word of the form $xxx$.  More generally, if
$u$ is a word with period $p$, then we say that $u$ is an
\emph{$\alpha$-power}, where $\alpha=|u|/p$.  An \emph{overlap} is a
word of the form $axaxa$, where $a$ is a letter and $x$ is a word
(possibly empty).  A word is \emph{squarefree} (resp.
\emph{cubefree}, \emph{overlap-free}) if none of its factors are
squares (resp. cubes, overlaps).  For any real number $\alpha$, we say
that an infinite word is \emph{$\alpha$-powerfree} if for all $\beta
\geq \alpha$, none of its factors are $\beta$-powers.  A
\emph{palindrome} is a word that equals its reversal.

Let $\mu$ be the \emph{Thue--Morse morphism} defined by $\mu(0)=01$, $\mu(1)=10$. Clearly $|\mu(u)|=2|u|$ for any factor $u$ of $m$.
We define the \emph{Thue-Morse word} as ${\bf m}=\mu^{\omega}(0)$.
If $u=x_1x_2\cdots x_s$ is a word over $\{0,1\}$ for some positive
integer $s$, then we define $\overline{u}$ by
$\overline{u}=y_1y_2\cdots y_s$ where $y_i=1-x_i$.

Let $x$ be a finite or infinite word.  A factor $v$ of $x$ is
\emph{left special} (resp.\ \emph{right special}) if there are
distinct letters $a$ and $b$ such that $va$ and $vb$
(resp.\ $av$ and $bv$) are factors of $x$.  A factor $v$ of
$x$ is \emph{bispecial} if it is both left special and right
special.  An infinite word is \emph{Sturmian} if it contains exactly
$n+1$ factors of length $n$ for every $n \geq 0$.  A Sturmian word is
\emph{standard} if each of its prefixes is left special.

Let $\phi$ be the \emph{Fibonacci morphism} defined by $\phi(0)=01$, $\phi(1)=0$.
We define the \emph{Fibonacci word} as ${\bf f}=\phi^{\omega}(0)$. We define $f_k=\phi^k(0)$.
We define the \emph{Fibonacci sequence} as $F_0=1$, $F_1=2$ and $F_k=F_{k-1}+F_{k-2}$ for $k\geq2$. 
Note that $|f_k|=F_k$ and that $f_k=f_{k-1}f_{k-2}$ 
(that is, $f_k$ is the concatenation of $f_{k-1}$ with $f_{k-2}$).
Also note that the Fibonacci word is a standard Sturmian word.

Let $\sigma$ be the \emph{Tribonacci morphism} defined by $\sigma(0)=01$, $\sigma(1)=02$, $\sigma(2)=0$.
We define the \emph{Tribonacci word} as ${\bf t}=\sigma^{\omega}(0)$. We define $t_k=\sigma^k(0)$.
We define the \emph{Tribonacci sequence} as $T_0=1$, $T_1=2$, $T_2=4$ and $T_k=T_{k-1}+T_{k-2}+T_{k-3}$ for $k\geq3$. Also, we define $t_{-1}=2$ and $T_{-1}=1$.
Note that $|t_k|=T_k$ and that $t_k=t_{k-1}t_{k-2}t_{k-3}$. We define $D_k=t_{k-1}t_{k-2}\cdots t_2t_1t_0$ for $k\geq1$. By convention, we define $D_0=\epsilon$.

\section{Some general properties of initial non-repetitive complexity}

Recall that the complexity function $c_{\bf w}(n)$ satisfies
$c_{\bf w}(n) > n$ for any aperiodic word ${\bf w}$.  This is not
necessarily true for the initial non-repetitive complexity function.
Nevertheless, the initial non-repetitive complexity must grow at least linearly
for any aperiodic word ${\bf w}$.  Note also that the initial non-repetitive
complexity is non-decreasing.

\begin{theorem}\label{sub_linear}
Let ${\bf w}$ be an infinite word and let $\varphi$ be the golden
ratio.  The following are equivalent.
\begin{enumerate}
\item ${\bf w}$ is ultimately periodic.

\item $\nsc_{\bf w}(n)$ is bounded.

\item $\displaystyle \limsup_{n \to \infty} \frac{\nsc_{\bf w}(n)}{n} = 0.$

\item $\displaystyle \limsup_{n \to \infty} \frac{\nsc_{\bf w}(n)}{n}
  < \frac{1}{1+\varphi^2}.$
\end{enumerate}
\end{theorem}

\begin{proof}
The implications $1 \Rightarrow 2 \Rightarrow 3 \Rightarrow 4$ are
straightforward.  We prove $4 \Rightarrow 1$.  Let $\varepsilon <
1/(1+\varphi^2)$ and suppose that there exists
$N$ such that $\nsc_{\bf w} (n) < \varepsilon n$ for all $n \geq N$.
Suppose further that $N$ satisfies
$\lceil(1+\varphi^2)\varepsilon(N+1)\rceil < N$.
For each $n \geq N$, there exist integers $i_n$ and $j_n$ satisfying
$0 \leq i_n < j_n \leq \varepsilon n$ such that
${\bf w}[i_n \ldots i_n+n-1] = {\bf w}[j_n \ldots j_n+n-1].$ Define
$p_n = j_n - i_n$ and note that ${\bf w}[i_n \ldots i_n+n-1]$ has
period $p_n \leq \varepsilon n$.  Define discrete intervals
$I_n = [i_n + \lceil \varphi^2p_n \rceil, i_n + n - 1]$.  For every $i \in I_n$, the prefix
${\bf w}[0 \ldots i-1]$ ends with a $\varphi^2$-power.  Moreover, since
\[
i_{n+1}+\lceil \varphi^2p_{n+1} \rceil \leq \varepsilon(n+1) + \lceil \varphi^2
\varepsilon(n+1) \rceil \leq \lceil(1+\varphi^2)\varepsilon(n+1)\rceil
+ 1 \leq  n \leq i_n+n,
\]
we have $\cup_{n \geq N} I_n = [i_N + \lceil \varphi^2p_N \rceil, \infty]$.
Thus, for every $i \geq i_N + \lceil \varphi^2p_N \rceil$,
the prefix ${\bf w}[0 \ldots i-1]$ ends
with a $\varphi^2$-power.  Mignosi, Restivo, and
Salemi \cite[Theorem~2]{MRS98} showed that this implies that ${\bf w}$
is ultimately periodic, as required.
\end{proof}

This result gives an interesting new characterization of ultimate
periodicity.  Later (Theorems~\ref{tm_nsc} and \ref{fib_nsc}) we shall
compute the initial non-repetitive complexity function for the Thue--Morse
word ${\bf m}$ and the Fibonacci word ${\bf f}$.  These results imply
\[
\limsup_{n \to \infty} \frac{\nsc_{\bf m}(n)}{n}  = 3
\]
and
\[
\limsup_{n \to \infty} \frac{\nsc_{\bf f}(n)}{n}  = 1.
\]
One may therefore reasonably wonder if the constant $1/(1+\varphi^2)$
is optimal in Theorem~\ref{sub_linear}, or if it could perhaps be
replaced by $1$.

Next we show that there are infinite words whose initial non-repetitive
complexity is maximal.  First, recall that for any alphabet of size
$q$ and any $n$ there exists a (non-cyclic) \emph{$q$-ary de~Bruijn
  sequence of order $n$}, that is, a word of length $q^n+n-1$
that contains every $q$-ary word of length $n$ as a factor (see \cite{Mar34}).  A
\emph{cyclic $q$-ary de~Bruijn sequence of order $n$} is a word $B_n$
of length $q^n$ that contains every $q$-ary word of length $n$ as a
\emph{circular factor}.  Here by circular factor we mean a factor of
some cyclic shift of $B_n$.

\begin{proposition}{\ }
\begin{itemize}
\item[(a)] Over any alphabet of size $ q \geq 3$ there exists an
  infinite word ${\bf w}$ satisfying $$\nsc_{\bf w}(n) = q^n$$ for all
  $n \geq 1$.
\item[(b)] Over a binary alphabet there exists an
  infinite word ${\bf w}$ satisfying $$\nsc_{\bf w}(2n) = 2^{2n}$$ for all
  $n \geq 1$.
\end{itemize}
\end{proposition}

\begin{proof}
This is a consequence of a result of Iv\'anyi (part (a) only)
\cite{Iva87} or Becher and Heiber \cite{BH11} (both parts).
They showed that over alphabets of size at least three $3$ any
(non-cyclic) de~Bruijn
sequence of order $n$ can be extended to a de~Bruijn sequence of order
$n+1$.  Taking the limit of such extensions gives an infinite word
with the desired property.  Curiously, over a binary alphabet,
de~Bruijn sequences of order $n$ cannot be extended to order $n+1$,
but can be extended to give de~Bruijn sequences of order $n+2$.
\end{proof}

Next we explore the relationship (if any) between the factor
complexity and initial non-repetitive complexity functions.  The next result
shows that there are infinite words with maximal factor complexity but
only linear initial non-repetitive complexity.

\begin{proposition}
Let $q>1$ and let $B_n$ denote a cyclic $q$-ary de~Bruijn sequence of order $n$
starting with $n$ $0$'s. Then
$${\bf x}=0^{q^{q^1}}B_10^{q^{q^2}}B_20^{q^{q^3}}B_3\cdots $$ is an infinite
word with complexity $q^n$ and initial non-repetitive complexity $\leq4n$ for
$n\geq1$.
\end{proposition}

\begin{proof}
  Since $B_n$ contains every $q$-ary word of length $n$ as a circular
  factor, having at least $n-1$ $0$'s follow each $B_n$ ensures that
  every $q$-ary word of length $n$ shows up in ${\bf x}$. Thus ${\bf x}$ has complexity
  $q^n$ for all positive $n$. The factor of length $n<q^{q^k}$ starting at the
  first position of the factor $0^{q^{q^k}}$ consists of $n$
  $0$'s. The factor of length $n$ starting at the  second $0$ of
  $0^{q^{q^k}}$ also consists of $n$ $0$'s. It
  follows that if $n<q^{q^k}$, then $\nsc_{\bf x}(n)$ must be less or
  equal to the length of the prefix of $x$ ending just before the second
  $0$ of the $0^{q^{q^k}}$ substring. That length is
  $(q^{q^{k-1}}+q^{q^{k-2}}+\cdots +q^{q^1})+(q^{k-1}+q^{k-2}+\cdots
  +q^1)+1$
  for $k\geq2$. It follows that if $q^{q^{k-1}}\leq n<q^{q^k}$, then

\begin{eqnarray*}
\nsc_{\bf x}(n)&\leq& \sum_{i=1}^{k-1} q^{q^i} + \sum_{i=1}^{k-1} q^i +1\\
&\leq& \sum_{i=1}^{q^{k-1}} q^i +  \sum_{i=1}^{k-1} q^i +1\\
&\leq& \frac{q-q^{q^{k-1}+1}}{1-q} +\frac{q-q^k}{1-q} +1\\
&\leq& \left(\frac{q}{q-1}\right)(q^{q^{k-1}} + q^{k-1}-2)+1\\
&\leq& 2(2n-2)+1\\
&\leq& 4n.
\end{eqnarray*}

It is clear that if $n<q^q$ then $\nsc_{\bf x}(n)=1$, which completes
the proof.
\end{proof}

The previous result showed that there can be a dramatic difference
between the behaviours of the factor complexity function and the
initial non-repetitive complexity function.  Next we show what kind of
separation is possible for these two functions when we restrict our
attention to pure morphic words.  It is well-known that pure morphic
words have $O(n^2)$ factor complexity \cite{AS03}.

Define the morphism $\phi$ by $\phi(0)=001$, $\phi(1)=1$ and let
${\bf x}=\phi^\omega(0)$.  It is known that ${\bf x}$ has
$\Theta(n^2)$ factor complexity (see, for instance,
\cite[Example~4.7.67]{BR10}).

\begin{lemma}\label{initial_square}
For all $k\geq0$, the word ${\bf x}$ has the prefix $zz$, where
$|z|=2^{k+1}-1$.
\end{lemma}

\begin{proof}
Since ${\bf x}$ begins with $00$, it begins with $\phi^k(0)\phi^k(0)$ for
all $k \geq 0$.  Thus we may take $z=\phi^k(0)$.  It remains to show
that $|z|=2^{k+1}-1$.

Let $M=\begin{bmatrix} 2&0\\1&1 \end{bmatrix}$ be the adjacency matrix
associated with $\phi$. Then an easy induction shows that
$M^k=\begin{bmatrix} 2^k&0\\2^k-1&1 \end{bmatrix}$ for $k\geq0$.
Now we have
$$|z|=|\phi^k(0)|=|\phi^k(0)|_0+|\phi^k(0)|_1=2^k+2^k-1=2^{k+1}-1,$$ as
required.
\end{proof}

\begin{proposition}
$\nsc_{\bf x}(n)<2n$ for $n\geq1$.
\end{proposition}

\begin{proof}
From Lemma~\ref{initial_square}, we have that if $n\leq2^{k+1}-1$,
then $\nsc_{\bf x}(n)\leq2^{k+1}-1$. It follows that if
$2^k-1<n\leq2^{k+1}-1$, then $\nsc_{\bf
  x}(n)\leq2^{k+1}-1\leq2(2^k)-1\leq2n-1<2n$.
\end{proof}

\section{Initial non-repetitive complexity of the Thue--Morse word}

We now begin to compute explicity the initial non-repetitive complexity
functions for some of the classical infinite words, beginning with the
Thue--Morse word.  Recall that the Thue--Morse word is the word ${\bf
  m} = \mu^\omega(0)$, where $\mu(0) = 01$ and $\mu(1) = 10$.

\begin{theorem}\label{tm_nsc}
If $2^{k-1}<n\leq2^k$ for $k\geq1$, then $\nsc_{\bf m}(n)=3(2^{k-1})$.
\end{theorem}

The proof of the theorem will follow from the following lemmas.
First note that it follows easily from the definition of $\mu$ that 
if $A$ is a prefix of ${\bf m}$ of length $2^k$ for $k\geq0$,
then $\mu(A)=A\overline{A}$.

\begin{proposition}\label{tm_lb}
If $n\leq2^k$ for $k\geq1$, then $\nsc_{\bf m}(n)\leq3(2^{k-1})$.
\end{proposition}

\begin{proof}
We show that ${\bf m}[0\ldots 2^k-1]={\bf m}[2^k+2^{k-1}\ldots
2^{k+1}+2^{k-1}-1]$.  Let $A={\bf m}[0\ldots 2^{k-1}-1]$.  Then
$\mu^3(A)=\mu^2(A\overline{A})=\mu(A\overline{A}\overline{A}A)=A\overline{A}\overline{A}A\overline{A}AA\overline{A}$
is a prefix of ${\bf m}$. 
Since $|A|=2^{k-1}$, then $A\overline{A}={\bf m}[0\ldots 2^k-1]$ and
$A\overline{A}={\bf m}[2^k+2^{k-1}\ldots 2^{k+1}+2^{k-1}-1]$.
The upper bound for $\nsc_{\bf m}(n)$ follows immediately.
\end{proof}

We make use of the next result to obtain a matching lower bound.

\begin{lemma} \cite[Example 10.10.3]{AS03}\label{distinct_factors_tm}
For any integer $k\geq2$, each factor of ${\bf m}$ of length $2^{k-1}+1$
occurs in the prefix of ${\bf m}$ of length $2^{k+1}$.
Furthermore, each one of these factors occurs exactly once in this prefix.
\end{lemma}

\begin{proposition}\label{tm_ub}
If $2^{k-1}<n$ for $k\geq2$, then $\nsc_{\bf m}(n)\geq3(2^{k-1})$.
\end{proposition}

\begin{proof}
By Lemma~\ref{distinct_factors_tm}, the first $2^{k+1}-(2^{k-1}+1)+1 =
3(2^{k-1})$ factors of length $2^{k-1}+1$ appearing in ${\bf m}$ are all
distinct.  Consequently the first $3(2^{k-1})$ length-$n$ factors
appearing in ${\bf m}$ must also be distinct.
\end{proof}

Using Propositions~\ref{tm_lb} and \ref{tm_ub} and that $\nsc_{\bf
  m}(2)=3$ (obtained through observation), we get Theorem~\ref{tm_nsc}
and thus the proof is complete. Though the theorem is not defined for
$n=1$, please note that $\nsc_{\bf m}(1)=2$.

\section{Initial non-repetitive complexity of the Fibonacci word}

Recall that the Fibonacci word is the word ${\bf
  f} = \phi^\omega(0)$, where $\phi(0) = 01$ and $\phi(1) = 0$.

\begin{theorem}\label{fib_nsc}
If $F_{k-1}\leq n+1<F_k$ for $k\geq2$, then $\nsc_{\bf f}(n)=F_{k-1}$.
\end{theorem}

We first need some preliminary results.  Recall that $f_k =
\phi^k(0)$.

\begin{lemma}\cite[Chapter~2]{Lot02}\label{fib_last2}
For $k\geq2$, the words $f_k=f_{k-1}f_{k-2}$ and $f_{k-2}f_{k-1}$ differ only by their last two letters.
\end{lemma}

\begin{proposition}\label{fib_ub}
If $n+1<F_k$ for $k\geq2$, then $\nsc_{\bf f}(n)\leq F_{k-1}$.
\end{proposition}

\begin{proof}
We show that for any positive integer $k\geq2$, $${\bf f}[0\ldots
F_{k}-3]={\bf f}[F_{k-1}\ldots F_{k+1}-3].$$
We know $f_{k+1}=f_kf_{k-1}=f_{k-1}f_{k-2}f_{k-1}$ is a prefix of $f$. 
By Lemma~\ref{fib_last2}, ${\bf f}[0\ldots F_{k}-1]$ and ${\bf f}[F_{k-1}\ldots F_{k+1}-1]$ agree up to but not including the last two positions. The result follows.
\end{proof}

Furthermore, the Fibonacci word is a standard Sturmian word, so
for $k\geq1$, $f_k=u_krs$, where $rs=01$ if $k$ is odd or $rs=10$ if
$k$ is even. The $u_k$'s are known as \emph{central words} and it is
known that these central words are palindromes and are bispecial
(see \cite[Chapter~2]{Lot02}).

A \emph{semicentral word} \cite{BDF13} is a word in which the longest repeated
prefix, longest repeated suffix, longest left special factor and
longest right special factor are all the same word. Furthermore, this
prefix/suffix/bispecial factor is a central word.

\begin{lemma}\label{semicentral} \cite[Proposition 16]{DF13}
The semicentral prefixes of a standard Sturmian word are precisely the
words of the form $u_krsu_k$ for $k\geq1$.
\end{lemma}

The property described in the next lemma is the property of having
\emph{grouped factors}, which was mentioned in the introduction.

\begin{lemma}\label{grouped_factors} \cite[Corollary 1]{Cas97}
A sequence is Sturmian if and only if, for $n\geq0$, it has a factor
of length $2n$ containing all factors of length $n$ exactly
once. Furthermore, if $n\geq1$, then there are exactly two such
factors of length $2n$, namely $w01v$ and $w10v$, where $w$ is the
unique right special factor of length $n-1$ and $v$ is the unique left
special factor of length $n-1$.
\end{lemma}

\begin{proposition}\label{fib_lb}
If $F_{k-1}\leq n+1$ for $k\geq2$, then $\nsc_{\bf f}(n)\geq F_{k-1}$.
\end{proposition}

\begin{proof}
It suffices to show that for $n=F_{k-1}-1$, the first $F_{k-1}$ factors
of ${\bf f}$ of length $n$ are all distinct.  We know from
Lemma~\ref{semicentral} that the Fibonacci word has the prefix
$u_{k-1}rsu_{k-1}$ where $rs=01$ or $rs=10$.  Since
these prefixes are of the same construction as the factors detailed
in Lemma~\ref{grouped_factors} ($u_{k-1}$ is the left and right
special factor of length $n-1$), and since
$$|u_{k-1}rsu_{k-1}| = 2|u_{k-1}r| = 2(F_{k-1}-1) = 2n,$$
it follows that this semicentral prefix contains all factors of length $n$ exactly
once.  Thus for all $n\geq F_{k-1}-1$, all factors of length
$n$ are distinct over the first $2(F_{k-1}-1)$ positions and so the
result follows.
\end{proof}

Using Propositions~\ref{fib_ub} and \ref{fib_lb}, we get
Theorem~\ref{fib_nsc} and thus the proof is complete.

\section{Initial non-repetitive complexity of the Tribonacci word}

Recall that the Tribonacci word is the word ${\bf
  t} = \sigma^\omega(0)$, where $\sigma(0) = 01$, $\sigma(1) = 02$,
and $\sigma(2) = 0$.

\begin{theorem}\label{trib_nsc}
If $\frac{T_k+T_{k-2}-3}{2}<n\leq\frac{T_{k+1}+T_{k-1}-3}{2}$ for $k\geq1$, then $\nsc_{\bf t}(n)=T_k$.
\end{theorem}

We first need to recall some known properties of the Tribonacci word.
Recall that $t_k = \sigma^k(0)$ and that $D_k=t_{k-1}t_{k-2}\cdots
t_2t_1t_0$ for $k\geq1$.

\begin{lemma}\label{TW2.5} \cite[Theorem 2.5]{TW07}
For $k\geq2$, the longest common prefix of $t_{k-3}t_{k-1}t_{k-2}$ and $t_k$ is $D_{k-2}$.
\end{lemma}

\begin{lemma}\label{TW2.9} \cite[Proposition 2.9]{TW07}
For $k\geq1$, $|D_k|=\frac{T_{k+1}+T_{k-1}-3}{2}$.
\end{lemma}

\begin{lemma}\label{trib_repeated_prefix}
For any positive integer $k\geq2$, $${\bf t}\left[0\ldots
  \frac{T_{k+1}+T_{k-1}-3}{2}-1\right]={\bf t}\left[T_k\ldots T_k+\frac{T_{k+1}+T_{k-1}-3}{2}-1\right].$$
\end{lemma}

\begin{proof}
We know
$t_{k+2}=t_{k+1}t_kt_{k-1}=t_kt_{k-1}t_{k-2}t_kt_{k-1}=t_{k-1}t_{k-2}t_{k-3}t_{k-1}t_{k-2}t_kt_{k-1}$
is a prefix of ${\bf t}$ for $k\geq2$. By Lemma~\ref{TW2.5}, we know that
$t_{k-3}t_{k-1}t_{k-2}$ agrees with $t_k$ up to the first $|D_{k-2}|$
symbols.  It follows that $t_{k-1}t_{k-2}t_{k-3}t_{k-1}t_{k-2}$ agrees
with $t_{k-1}t_{k-2}t_k$ up to the first $|D_k|$ symbols.
Since $t_{k-1}t_{k-2}t_k={\bf t}[T_k\ldots T_k+T_{k+1}-1]$, the result
follows from Lemma~\ref{TW2.9}.
\end{proof}

We therefore have the following.

\begin{proposition}\label{trib_ub}
If $n\leq\frac{T_{k+1}+T_{k-1}-3}{2}$ for $k\geq2$, then $\nsc_{\bf t}(n)\leq T_k$.
\end{proposition}

Before proving the lower bound for $\nsc_{\bf t}(n)$, we need some
additional properties of the Tribonacci word.

\begin{lemma}\label{trib_bispecial} \cite[Proof of Proposition 3.3]{RSZ10}
The bispecial factors of ${\bf t}$ are precisely the palindromic
prefixes of ${\bf t}$. Furthermore, the lengths of these (nonempty) prefixes are $\frac{T_{k+2}+T_k-3}{2}$ for $k\geq0$.
\end{lemma}

\begin{lemma}\label{trib_pal_image} \cite[Lemma 2.3]{TW07}
If $w$ is a palindrome, then $\sigma(w)0$ is a palindrome.
\end{lemma}

\begin{lemma}\label{trib_pal_length}
If $w$ is a palindromic prefix of ${\bf t}$ of length $|D_k|$ for $k\geq1$,
then $\sigma(w)0$ is a palindromic prefix of ${\bf t}$ of length
$|D_{k+1}|$.
\end{lemma}

\begin{proof}
We know from Lemma~\ref{trib_bispecial} that all palindromic prefixes
of ${\bf t}$ are of length $|D_k|$ for $k\geq1$. If $w$ is a palindromic
prefix of ${\bf t}$ of length $|D_k|$, then clearly $\sigma(w)$ is a prefix
of ${\bf t}$. Furthermore, since $w$ starts with a 0 and is a palindrome, it
ends with a 0. So $\sigma(w)$ ends with a 1. Since strings 11 and 12 are
not in ${\bf t}$, then $\sigma(w)$ must be followed by a 0. Thus,
$\sigma(w)0$ is a prefix of ${\bf t}$ and we know from
Lemma~\ref{trib_pal_image} that it is a palindrome. 
Applying the morphism $\sigma$ to $w$ will at most double the length. 
Thus
\begin{align*}
|\sigma(w)0| &\leq2|w|+1\\
&=2\left(\frac{T_{k+1}+T_{k-1}-3}{2}\right)+1\\
&=\frac{2T_{k+1}+2T_{k-1}-4}{2}\\
&<\frac{T_{k+2}+T_{k+1}+T_k+T_k+T_{k-1}+T_{k-2}-3}{2}\\
&=\frac{T_{k+3}+T_{k+1}-3}{2}\\
&=|D_{k+2}|.
\end{align*}
 So the only option for the length of $\sigma(w)0$ is $|D_{k+1}|$.
\end{proof}

\begin{lemma}\label{trib_pal_follow}
If $w$ is a (nonempty) palindromic prefix of ${\bf t}$, then the first
symbols that follow each of the first two occurrences of $w$ in ${\bf t}$
are different.
\end{lemma}

\begin{proof}
By induction on $k$ where $|w|=|D_k|$. Since ${\bf t}=0102\cdots $, the
result holds for $k=1$. Assume that the first symbol that follows each
of the first two occurences of $w$ are different where $|w|=|D_k|$.
Since $21$ and $22$ are not factors of ${\bf t}$, each of these
occurrences of $w$ are followed by different words among $0$, $1$, and
$20$.  Now, since $\sigma(w0)=\sigma(w)01$, $\sigma(w1)=\sigma(w)02$,
and $\sigma(w20)=\sigma(w)001$, we see that the first two occurences of
$\sigma(w)0$ are followed by different symbols. Since
$|\sigma(w)0|=|D_{k+1}|$ by Lemma~\ref{trib_pal_length}, this implies
that the statement holds for $k+1$ and thus the statement holds for
all $k$ by induction.
\end{proof}

The following is a well-known property of ${\bf t}$.

\begin{lemma}\label{trib_unique_special}
There is a unique left special factor and a unique right special
factor of each length in ${\bf t}$.
\end{lemma}

\begin{lemma}\label{trib_distinct_factors}
Let $v$ denote the prefix of length $|D_{k-1}|$ of ${\bf t}$ for
$k\geq2$. All the factors of length $|D_{k-1}|$ that start between the
beginning of the first occurrence of $v$ and the beginning of the
third occurrence of $v$ are distinct (except for $v$).
\end{lemma}

\begin{proof}
Firstly, since ${\bf t}$ is recurrent, we know that there are three
occurrences of $v$ in ${\bf t}$. For the sake of contradiction, assume
the factor $u(\neq v)$ of length $|D_{k-1}|$ has two occurences in
${\bf t}$ before we reach the first symbol of the third occurrence of $v$. For simplicity, let $v_j$ denote the $j$th occurrence of $v$ and $u_i$ the $i$th occurrence of $u$. If the starting symbol of $u_i$ is between the starting symbol of $v_j$ and $v_{j+1}$, then we will denote that by $v_j<u_i<v_{j+1}$. 

Case 1: $v_1<u_1<u_2<v_2$.

If $u_1$ and $u_2$ are preceded by different symbols, then $u$ is a
left special factor. This is a contradiction since $u\neq v$ and $v$
is the unique left special factor of length $|D_{k-1}|$ in ${\bf t}$. Thus, assume they are preceded by the same symbol. Then we obtain another factor (formed by the first $|D_{k-1}|-1$ symbols of $u$ and the symbol preceding $u_1$), which we will call $r$, of length $|D_{k-1}|$ such that $v_1<r_1<r_2<v_2$. Once again, if $r_1$ and $r_2$ are preceded by different symbols then we obtain a contradiction. By repeating this argument we eventually find that $v_1<v_j<u_2$ for some $j$, which contradicts our original assumption.

Case 2: $v_2<u_1<u_2<v_3$.

Similar to Case 1. 

Case 3: $v_1<u_1<v_2<u_2<v_3$.

We apply the same argument as in Case~1.  We either obtain the same contradiction described in that case, or we find that the factor starting with the first symbol of $v_1$ and ending with the last symbol of $u_1$ is identical to the factor starting with the first symbol of $v_2$ and ending with the last symbol of $u_2$. This is a contradiction since the symbols following $v_1$ and $v_2$ are different by Lemma~\ref{trib_pal_follow}.

In all three cases we obtain a contradiction. Thus all the factors of
length $|D_{k-1}|$ (except $v$) are distinct.
\end{proof}

\begin{lemma}\label{trib_distinct_factors2}
Let $k \geq 2$.  All factors of ${\bf t}$ of length $|D_{k-1}|+1$ that begin
prior to the third occurrence of the prefix of ${\bf t}$ of length
$|D_{k-1}|$ are distinct.
\end{lemma}

\begin{proof}
It is a direct result of Lemmas~\ref{trib_pal_follow} and
\ref{trib_distinct_factors}.
\end{proof}

\begin{lemma}\label{trib_squares} \cite[Section 6.3.5]{Gle06}
If a square $xx$ is a factor of ${\bf t}$, then $|x| \in \{T_k,T_k+T_{k-1}\}$
for some $k\geq1$.
\end{lemma}

\begin{lemma}\label{trib_overlap}
If a word $v$ of length $|D_{k-1}|$ for $k\geq5$ overlaps itself in
${\bf t}$, then the shortest period of $v$ is at least $T_{k-2}$.
\end{lemma}

\begin{proof}
The largest Tribonacci number or sum of consecutive Tribonacci numbers
less than $T_{k-2}$ is $T_{k-3}+T_{k-4}$. Let $v=xax$ be a factor of
${\bf t}$ of length $|D_{k-1}|$, where $x$ is a nonempty factor of ${\bf t}$ and
$a$ is a possibly empty factor of ${\bf t}$.  Suppose that ${\bf t}$ contains the
overlap $xaxax$. Note that $|xa|$ is a period of $v$. Also,
$|xa|<|D_{k-1}|$ and $2|xa|\geq|D_{k-1}|$. However,
$|D_{k-1}|=T_{k-2}+T_{k-3}+T_{k-4}+\cdots+T_0
=2T_{k-3}+2T_{k-4}+2T_{k-5}+T_{k-6}\cdots+T_0 >2(T_{k-3}+T_{k-4})$ for
$k\geq5$. So every period of $v$ must be larger than
$T_{k-3}+T_{k-4}$. Thus, from Lemma~\ref{trib_squares}, the shortest
period of $v$ is at least $T_{k-2}$.
\end{proof}

\begin{lemma}\label{trib_pal_positions}
If $v$ is a prefix of ${\bf t}$ of length $|D_{k-1}|$ for $k\geq2$, then the
second occurence of $v$ occurs at position $T_{k-1}$ and the third occurs at
position $T_k$.
\end{lemma}

\begin{proof}
Since ${\bf t}=01020100102010102010\cdots $, it can be observed that the
statement holds for $k=2,3,4$. Thus, assume for the rest of this proof
that $k\geq5$. By Lemma~\ref{trib_repeated_prefix}, we already know
that the prefix $v$ occurs at position $T_{k-1}$ and position
$T_k$. If there were an occurrence of $v$ that started somewhere between
the beginning of ${\bf t}$ and position $T_{k-1}$ of ${\bf t}$, then
by Lemma~\ref{trib_overlap} (note that $|D_{k-1}| > T_{k-1}$),
the start of this occurrence of $v$ must be at distance
at least $T_{k-2}$ from the beginning of ${\bf t}$ and at distance at least
$T_{k-2}$ from position $T_{k-1}$. This implies that $2T_{k-2}\leq T_{k-1}$ but
to the contrary we have
$2T_{k-2}=T_{k-2}+T_{k-3}+T_{k-4}+T_{k-5}=T_{k-1}+T_{k-5}>T_{k-1}$. Furthermore,
$T_{k-1}+2T_{k-2}>T_{k-1}+T_{k-2}+T_{k-3}=T_k$. It follows that no
occurrence of $v$ can start between the beginning of ${\bf t}$ and position
$T_{k-1}$ nor can it start anywhere between positions $T_{k-1}$ and $T_k$.
\end{proof}

\begin{proposition}\label{trib_lb}
If $\frac{T_k+T_{k-2}-3}{2}<n$ for $k\geq2$, then $\nsc_{\bf t}(n)\geq T_k$.
\end{proposition}

\begin{proof}
The result follows from Lemmas~\ref{trib_distinct_factors2} and
\ref{trib_pal_positions}.
\end{proof}

Using Propositions~\ref{trib_ub} and \ref{trib_lb} and that $\nsc_{\bf
  t}(1)=2$ (making the theorem hold for $k=1$), we get
Theorem~\ref{trib_nsc} and thus the proof is complete.

\section{Initial non-repetitive complexity of squarefree words}\label{sec_sqf}

In this section we examine the possible behaviour of the
initial non-repetitive complexity function for words avoiding squares
or cubes.  In particular, we attempt to construct words that
avoid the desired type of repetition but have initial non-repetitive
complexity as low as possible.

\begin{proposition}~
\begin{enumerate}
\item There is no infinite squarefree word ${\bf x}$ that has $\nsc_{\bf x}(n)<2n$ for all $n$.
\item There is no infinite cubefree word ${\bf x}$ that has $\nsc_{\bf
    x}(n)<\frac{3}{2}n$ for all $n$.
\end{enumerate}
\end{proposition}

\begin{proof}
1. If $\nsc_{\bf x}(n)<2n$ then $\nsc_{\bf x}(1)=1$ and therefore $x=aa\cdots$, a contradiction.

2. If $\nsc_{\bf x}(n)<\frac{3}{2}n$ then $\nsc_{\bf x}(1)=1$ and $\nsc_{\bf x}(2)\leq2$. It follows that $x=aaa\cdots$ which is a contradiction.
\end{proof}

Consider the infinite alphabet $\Sigma = \{0,1,2,\ldots\}$.  We define
the sequence of \emph{Zimin words}, $Z_0, Z_1, Z_2, \ldots$,
as follows:  $Z_0 = \epsilon$ and $Z_{n+1} = Z_nnZ_n$ for $n
\geq 0$.  Let
\[
{\bf x} = 0102010301020104 \cdots,
\]
also known as the \emph{ruler sequence}, be the limit of the $Z_n$.

\begin{theorem}\label{zimin_nsc}
The infinite word ${\bf x}$ is squarefree and satisfies
$n<\nsc_{\bf x}(n) \leq 2n$ for all $n \geq 1$.
\end{theorem}

\begin{proof}
For the squarefreeness of ${\bf x}$ see \cite{GS09}.
By the definition of ${\bf x}$, if $n\leq2^k-1$, then $\nsc_{\bf
  x}(n)\leq2^k$ for $k\geq1$. It follows that if $2^{k-1}\leq n<2^k$,
then $\nsc_{\bf x}(n)\leq2^k\leq2(2^{k-1})\leq2n$ for all $n$. Also,
since ${\bf x}$ is square-free, clearly $n< \nsc_{\bf x}(n)$ for all $n$.
\end{proof}

So we can obtain an infinite squarefree word over an infinite
alphabet that has $\nsc_{\bf x}(n) \leq 2n$ for all $n$. Furthermore,
for this word there are infinitely many values of $n$ such that
$\nsc_{\bf x}(n)=2n$.

Using an infinite alphabet may seem like ``cheating'', so next we
examine what can be done over a finite alphabet.
We will make use of a morphism $\theta : \{0,1,2,\ldots\}^* \to
\{a,b,c,d,e\}^*$, which maps squarefree words over an infinite alphabet
to squarefree words over an alphabet of size $5$.  First, let
\[
{\bf w} = abcacbabcbacabc\cdots
\]
be the well-known squarefree word obtained by iterating the morphism
\[
a \to abc, \quad\quad b \to ac, \quad\quad c \to b.
\]
For $i \geq 0$, let $W_i$ be the prefix of ${\bf w}$ of length
$i$.  We define $\theta(i)=dW_ieW_i$ for all $i\geq0$.  The
map $\theta$ is squarefree \cite[Corollary~1.4]{BEM79}; that is, if
$u$ is squarefree, then $\theta(u)$ is squarefree.

\begin{theorem}
Let ${\bf x}$ be the ruler sequence defined previously.  Then ${\bf
  y}=\theta({\bf x})$ is a square-free word with $\nsc_{\bf y}(n)<3n$ for
all $n$ except $n=2$.
\end{theorem}

\begin{proof}
It is relatively easy to see that the prefix $A$ of ${\bf x}$ of length
$2^k-1$ will have $2^{k-1}$ $0$'s, $2^{k-2}$ $1$'s, and so on,
down to only one occurrence of $k-1$. Furthermore, as a result of
how we defined the $W_i$'s, we have $|\theta(i)|=2(i+1)$. Thus, we have
\begin{align*}
|\theta(A)|&=2k+2(2(k-1))+4(2(k-2))+\cdots +2^{k-1}(2(1))\\
&=\sum_{i=0}^{k-1} 2^{i+1}(k-i)\\
&=2^{k+2}-2k-4.
\end{align*}
Furthermore, if $B$ is the prefix of ${\bf x}$ of
length $2^k$, then
$$|\theta(B)|=|\theta(A)|+2(k+1)=2^{k+2}-2k-4+2k+2=2^{k+2}-2,$$ since
$|\theta(k)|=2(k+1)$. As a result of the fact (see proof of
Theorem~\ref{zimin_nsc}) that $\nsc_{\bf
  x}(n)\leq2^k$ for $2^{k-1}\leq n<2^k$, if
$2^{k+1}-2(k-1)-4<n\leq2^{k+2}-2k-4$, then $\nsc_{\bf
  y}(n)\leq2^{k+2}-2$. The expression
$$\frac{2^{k+2}-2}{2^{k+1}-2(k-1)-3}$$ is a decreasing function of $k$
and is less than 3 for $k\geq4$. Along with the fact that $\nsc_{\bf
  y}(n)<3n$ for $n\leq22$ (other than $n=2$), which can be obtained
through computation, we have $\nsc_{\bf y}(n)<3n$ for all $n$ except
$n=2$.
\end{proof}

It should be noted that
\[
\lim_{k \to \infty} \frac{2^{k+2}-2}{2^{k+2}-2k-4}=1
\quad\quad\text{ and }\quad\quad
\lim_{k \to \infty} \frac{2^{k+2}-2}{2^{k+1}-2(k-1)-3}=2,
\]
meaning that
\[
\liminf_{n \to \infty} \frac{\nsc_{\bf y}(n)}{n}=1
\quad\quad\text{ and }\quad\quad
\limsup_{n \to \infty} \frac{\nsc_{\bf y}(n)}{n} \leq 2.
\]

To obtain a result over a $3$-letter alphabet we will need a morphism
$\sigma$ (found by Brandenburg
\cite[Theorem~4]{Bra83}), which maps squarefree words on
$\{a,b,c,d,e\}$ to squarefree words on $\{a,b,c\}$. We define it by
\begin{align*}
\sigma(a)&=abacabcacbabcbacbc\\
\sigma(b)&=abacabcacbacabacbc\\
\sigma(c)&=abacabcacbcabcbabc\\
\sigma(d)&=abacabcbacabacbabc\\
\sigma(e)&=abacabcbacbcacbabc.
\end{align*}

\begin{theorem}
The word ${\bf z}=\sigma({\bf y})$ is square-free and has $\nsc_{\bf
  z}(n) <3n$ for all $n>36$.
\end{theorem}

\begin{proof}
Since $|\sigma(m)|=18$ for $m\in\{a,b,c,d,e\}$ and $\nsc_{\bf
  y}(n)\leq2^{k+2}-2$ if $$2^{k+1}-2(k-1)-4<n\leq2^{k+2}-2k-4,$$ then it
follows that if $$18(2^{k+1}-2(k-1)-4)<n\leq18(2^{k+2}-2k-4),$$ then
$\nsc_{\bf z}(n)\leq18(2^{k+2}-2)$. The expression
$$\frac{18(2^{k+2}-2)}{18(2^{k+1}-2(k-1)-4)+1}$$ is a decreasing
function of $k$ and is less than 3 for $k\geq4$. Along with the fact
that  $\nsc_{\bf z}(n)\leq3n$ for $36 < n \leq(18)(22)=396$, which can
be obtained through computation, $\nsc_{\bf z}(n)<3n$ for all $n>36$.
\end{proof}

Furthermore, similar to the word ${\bf y}$, we have
\[
\lim_{k \to \infty} \frac{18(2^{k+2}-2)}{18(2^{k+2}-2k-4)}=1
\quad\quad \text{ and }\quad\quad 
\lim_{k \to \infty} \frac{18(2^{k+2}-2)}{18(2^{k+1}-2(k-1)-4)+1}=2,
\]
meaning that
\[
\liminf_{n \to \infty} \frac{\nsc_{\bf z}(n)}{n}=1
\quad\quad\text{ and }\quad\quad
\limsup_{n \to \infty} \frac{\nsc_{\bf z}(n)}{n} \leq 2.
\]

Also note that since the Thue--Morse word is overlap-free,
Theorem~\ref{tm_nsc} shows
that it is an example of an overlap-free word with $\nsc_{\bf m}(n) <
3n$ for all $n \geq 1$.

\section{Open questions}

\begin{question}
  Is the constant $1/(1+\varphi^2)$ in Theorem~\ref{sub_linear} best
  possible?  Can it be replaced by $1$?
\end{question}

\begin{question}
For each positive integer $d$, it is possible to construct an infinite
word ${\bf x}$ whose initial non-repetitive complexity is $\Theta(n^d)$?  What
are the possibilities for the usual factor complexity of such a word?
\end{question}

\begin{question}
Is the word ${\bf x}$ of Theorem~\ref{zimin_nsc} the only (up to
permutation of the infinite alphabet) infinite squarefree word such that
$\nsc_{\bf x}(n)\leq2n$ for all $n$?
\end{question}

\begin{question}
Are the examples given in Section~\ref{sec_sqf}
optimal for squarefree words:  i.e., are
there squarefree words whose initial non-repetitive
complexity functions grow even slower than the examples given here?
\end{question}

\begin{question}
Can results similar to those proved here also be proved for the
function $\mbox{nrc}_{\bf x}(n)$ defined in the Introduction?  A
detailed study of this function would be quite interesting.
\end{question}

\section{Acknowledgments}
The first author is supported by an NSERC USRA, the second by NSERC
Discovery Grant \#418646-2012.  We thank all of the anonymous
referees:  each of them had several helpful suggestions that improved
the quality of this paper.

\end{document}